\documentclass[12pt]{amsart}

\usepackage{amsfonts}
\usepackage{hyperref}
\usepackage{amsmath,amssymb,amsthm}
\usepackage{enumerate,graphicx}
\usepackage[square,numbers]{natbib}
\usepackage{blkarray}
\usepackage{pstricks-add}
\usepackage{auto-pst-pdf,pst-tree}

\newtheorem{theorem}{Theorem}[section]

\newtheorem{proposition}[theorem]{Proposition}
\newtheorem{corollary}[theorem]{Corollary}
\newtheorem{lemma}[theorem]{Lemma}
\theoremstyle{definition}

\newtheorem{definition}[theorem]{Definition}

\newtheorem{remark}[theorem]{Remark}

\begin{document}
\title{On the escape rate of unique beta-expansions}
\author{Jung-Chao Ban}
\address[Jung-Chao Ban]{Department of Applied Mathematics, National Dong Hwa
University, Hualien 970003, Taiwan, R.O.C.}
\email{jcban@mail.ndhu.edu.tw}
\author{Chih-Hung Chang*}
\thanks{*To whom correspondence should be addressed}
\address[Chih-Hung Chang]{Department of Applied Mathematics, National
University of Kaohsiung, Kaohsiung 81148, Taiwan, R.O.C.}
\email{chchang@nuk.edu.tw}
\author{Bing Li}
\address[Bing Li]{Department of Mathematics, South China University of
Technology, Guangzhou 510640, PR China}
\email{libing0826@gmail.com}
\date{}

\begin{abstract}
Let $1<\beta \leq 2$. It is well-known that the set of points in $%
[0,1/(\beta -1)]$ having unique $\beta $-expansion, in other words, those
points whose orbits under greedy $\beta $-transformation escape a hole
depending on $\beta $, is of zero Lebesgue measure. The corresponding escape
rate is investigated in this paper. A formula which links the Hausdorff
dimension of univoque set and escape rate is established in this study. Then
we also proved that such rate forms a devil's staircase function with
respect to $\beta $.
\end{abstract}
\maketitle

\baselineskip=1.2 \baselineskip



\section{Introduction}

\subsection{History and motivation: beta shifts}

There are many ways to represent real numbers such as decimal expansion,
binary expansion, etc. In 1957, R\'{e}nyi \cite{Renyi-AMH1957} generalized
the expansions with integer bases to any base $\beta >1$ including
non-integer bases.

Let $1<\beta \leq 2$ and $x\geq 0$. Write 
\begin{equation}
x=\sum_{n=1}^{\infty }\frac{\varepsilon _{n}}{\beta ^{n}}\ \ \text{with}\ \
\varepsilon _{n}\in \{0,1\}\ \text{for all}\ n\geq 1\text{.}
\label{expression}
\end{equation}%
We call this expression or the sequence $(\varepsilon _{1},\cdots
,\varepsilon _{n},\cdots )$ a $\beta $\emph{-expansion of }$x$. It is
obvious that the smallest number which can be represented by (\ref%
{expression}) is $0$ and the largest one is $\frac{1}{\beta -1}$.

Let $J_{\beta }=[0,\frac{1}{\beta -1}]$ and 
\begin{equation*}
\Omega _{\beta }=\{0,1\}^{\mathbb{N}}:=\{(w_{1},\cdots ,w_{n},\cdots
):w_{n}=0\ \text{or}\ 1\ \text{for all}\ n\geq 1\}.
\end{equation*}%
Denote by $\Omega _{\beta }^{n}=\{0,1\}^{n}$ and $\Omega _{\beta }^{\ast
}=\cup _{n=1}^{\infty }\Omega _{\beta }^{n}$. Let $\prec $ and $\preceq $ be
the \emph{lexicographical order} on $\Omega _{\beta }$. More precisely, $%
w\prec w^{\prime }$ means that there exists $k\in \mathbb{N}$ such that $%
w_{i}=w_{i}^{\prime }$ for all $1\leq i<k$ and $w_{k}<w_{k}^{\prime }$,
meanwhile, $w\preceq w^{\prime }$ means that $w\prec w^{\prime }$ or $%
w=w^{\prime }.$ This order can be extended to $\Omega _{\beta }^{\ast }$ by
identifying a finite block $(w_{1},\dots ,w_{n})$ with the sequence $%
(w_{1},\dots ,w_{n},0^{\infty })$. The topic on the number of $\beta $%
-expansions of reals in $J_{\beta }$ is very popular since Erd\"{o}s
investigated the reals with unique $\beta $-expansion in 1990s. \label{1Here}

\begin{theorem}[\protect\cite{EJK-BdlSMdF1990}]
\label{EJK90} If $1<\beta <\frac{1+\sqrt{5}}{2}$, then every $x\in (0,\frac{1%
}{\beta -1})$ has a continuum of different $\beta $-expansions.
\end{theorem}

\begin{theorem}[\protect\cite{Sidorov-AMM2003}]
\label{Sidorov03} If $\frac{1+\sqrt{5}}{2}\leq \beta <2$, then $\lambda $%
-almost every $x\in (0,\frac{1}{\beta -1})$ has a continuum of different $%
\beta $-expansion, where $\lambda $ is the Lebesgue measure on $\mathbb{R}$.
\end{theorem}

Define a \emph{projection map} $\pi _{\beta }:\Omega _{\beta }\rightarrow
J_{\beta }$ as 
\begin{equation*}
\pi _{\beta }(w)=\sum_{n=1}^{\infty }\frac{w_{n}}{\beta ^{n}}
\end{equation*}%
for $w=(w_{1},\cdots ,w_{n},\cdots )\in \Omega _{\beta }$. Then $\#\pi
_{\beta }^{-1}(x)$ is the number of $\beta $-expansions of $x\in J_{\beta }$%
, here $\#$ denotes the cardinality of a finite set. Denote 
\begin{equation*}
\mathcal{U}_{\beta }=\{x\in J_{\beta }:\#\pi _{\beta }^{-1}(x)=1\},
\end{equation*}%
that is, the set of the points with unique $\beta $-expansion. The set $%
\mathcal{U}_{\beta }$ is called the \emph{univoque set}. Together with
Theorem \ref{EJK90} and Theorem \ref{Sidorov03}, we know that $\lambda (%
\mathcal{U}_{\beta })=0$ for any $1<\beta \leq 2$. Glendinning and Sidorov 
\cite{GS-MRL2001} showed a finer description on $\mathcal{U}_{\beta }$ as
the following.

\begin{theorem}[\protect\cite{GS-MRL2001}]
The set $\mathcal{U}_{\beta }$ is

\begin{itemize}
\item empty if $\beta\in (1, \frac{1+\sqrt{5}}{2}]$;

\item countable for $\beta \in (\frac{1+\sqrt{5}}{2},\beta _{\ast })$, where 
$\beta _{\ast }=1.787231650\dots $ is the Komornik-Loreti constant (see also 
\cite{KL-AMM1998});

\item an uncountable Cantor set of zero Hausdorff dimension if $%
\beta=\beta_* $; and

\item a set of positive Hausdorff dimension for $\beta \in (\beta _{\ast
},2) $.
\end{itemize}
\end{theorem}

Recently, Komornik, Kong and Li \cite{KKL-apa2015} gave much more
information for $\mathcal{U}_{\beta }$ for $\beta _{\ast }<\beta <2$. The
authors showed that the Hausdorff dimension of such set forms a devil's
staircase function, i.e. $D(\beta ):=\dim _{H}\mathcal{U}_{\beta }$ is
continuous, monotonic and $D^{\prime }<0$ almost everywhere. They also gave
explicit formula for the Hausdorff dimension of $\mathcal{U}_{\beta }$ when $%
\beta $ is in any admissible interval $[\beta _{L},\beta _{U}]$ (see Theorem
2.6 \cite{KL-N2015} for the definition of such interval).

Among all the $\beta $-expansions of any given number $x\in J_{\beta }$, the
maximum and minimum in the sense of lexicographical order are provided by
greedy and lazy algorithms respectively. These two algorithms can be induced
by greedy and lazy $\beta $-transformations respectively.

\begin{definition}[Greedy $\protect\beta $-transformation]
Let $1<\beta \leq 2$. The \emph{greedy }$\beta $\emph{-transformation} $%
G_{\beta }:J_{\beta }\rightarrow J_{\beta }$ is defined as 
\begin{equation*}
G_{\beta }(x)=%
\begin{cases}
\beta x,\ \ \ \  & \text{if}\ 0\leq x<\frac{1}{\beta }; \\ 
\beta x-1,\ \ \  & \text{if}\ \frac{1}{\beta }\leq x\leq \frac{1}{\beta -1}.%
\end{cases}%
\end{equation*}%
The system $(J_{\beta },G_{\beta })$ is called the \emph{greedy }$\beta $%
\emph{-transformation dynamical system}.
\end{definition}

A \emph{coding} of any $x\in J_{\beta }$ according to $G_{\beta }$ which
includes two branches can be given as follows. Define 
\begin{equation*}
\varepsilon _{1}(x,\beta )=%
\begin{cases}
0,\ \ \ \  & \text{if}\ 0\leq x<\frac{1}{\beta }; \\ 
1,\ \ \ \  & \text{if}\ \frac{1}{\beta }\leq x\leq \frac{1}{\beta -1}%
\end{cases}%
\end{equation*}%
for $x\in J_{\beta }$. That is, the first branch $[0,\frac{1}{\beta })$ of $%
G_{\beta }$ is labelled by 0 and the other one $[\frac{1}{\beta },\frac{1}{%
\beta -1})$ is labelled by 1. Denote $\varepsilon _{n}(x,\beta
):=\varepsilon _{1}(G_{\beta }^{n-1}x,\beta ).$ Then, 
\begin{equation}
x=\sum_{n=1}^{\infty }\varepsilon _{n}(x,\beta )\beta ^{-n},
\label{betaexpansion}
\end{equation}%
which is called the \emph{greedy }$\beta $\emph{-expansion} of $x\in
J_{\beta }$ and denote $\varepsilon (x,\beta )=(\varepsilon _{1}(x,\beta
),\cdots ,\allowbreak \varepsilon _{n}(x,\beta ),\cdots )$. The trapping
region of $(J_{\beta },G_{\beta })$ is $(I,T_{\beta })$, where $I=[0,1)$ and 
$T_{\beta }=G_{\beta }|_{I}$. More precisely, $T_{\beta }I=I$ and for any $%
1\leq x\leq \frac{1}{\beta -1}$, there is some $n\in \mathbb{N}$ such that $%
T_{\beta }^{n}(x)\in I$. If $1\leq x\leq \frac{1}{\beta -1}$, then $%
\varepsilon (x,\beta )=(1^{k},\varepsilon (y,\beta ))$ for some $k\geq 1$
and $y=G_{\beta }^{k}(x)\in I$.

\begin{definition}[Lazy $\protect\beta $-transformation]
The lazy $\beta $-transformation $L_{\beta }:J_{\beta }\rightarrow J_{\beta
} $ is defined as 
\begin{equation*}
L_{\beta }(x)=%
\begin{cases}
\beta x,\ \ \ \ \  & \text{if}\ 0\leq x\leq \frac{1}{\beta (\beta -1)}; \\ 
&  \\ 
\beta x-1,\ \ \ \  & \text{if}\ \frac{1}{\beta (\beta -1)}<x\leq \frac{1}{%
\beta -1}.%
\end{cases}%
\end{equation*}
\end{definition}

Besides the greedy and lazy $\beta $-expansions, other $\beta $-expansions
are called \emph{intermediate }$\beta $\emph{-expansion}, which is rich from
Theorem \ref{EJK90} and Theorem \ref{Sidorov03}. The corresponding $\beta $%
-transformations are also studied recently, for example, see \cite{KL-N2015}.

By the greedy and lazy $\beta $-transformations, the whole interval $%
J_{\beta }$ is partitioned to three parts: 
\begin{eqnarray*}
I_{0} &=&\left[ 0,1/\beta \right) \text{,}\ \Delta _{\beta }=\left[ 1/\beta
,1/\beta (\beta -1)\right] \text{, } \\
I_{1} &=&\left( 1/\beta (\beta -1),1/(\beta -1)\right] \text{.}
\end{eqnarray*}%
If $x\in I_{0}$, the first digit of both greedy and lazy expansions of $x$
are 0; since the greedy and lazy expansions are the maximum and minimum
among all $\beta $-expansions, we know that the first digit of any $\beta $%
-expansion of $x$ is 0. Similarly, if $x\in I_{1}$, then the first digit of
all $\beta $-expansions of $x$ are 1. For $x\in \Delta _{\beta }$, since the
first digit of greedy expansion is 1 and the lazy is 0, the first digits of $%
\beta $-expansions of $x$ have two possibilities.

\begin{proposition}
\label{Prop: 1}The univoque set 
\begin{eqnarray*}
\mathcal{U}_{\beta } &=&\{x\in J_{\beta }:G_{\beta }^{n}(x)\notin \Delta
_{\beta }\ \forall n\geq 0\} \\
&=&\{x\in J_{\beta }:L_{\beta }^{n}(x)\notin \Delta _{\beta }\ \forall n\geq
0\}\text{.}
\end{eqnarray*}
\end{proposition}

This proposition tells us that the set $\mathcal{U}_{\beta }$ consists of
the points whose orbit under $G_{\beta }$ or $L_{\beta }$ will never fall in
the hole $\Delta _{\beta }$. So $\mathcal{U}_{\beta }$ is regarded as a
problem of \emph{open system} (\emph{dynamical system with hole }and \emph{%
exclusion systems}) and we just need to focus on the greedy $\beta $%
-transformation $G_{\beta }$.

\subsection{History and motivation: open systems and escape rate}

Proposition \ref{Prop: 1} reveals that the study of the topological
properties of $\mathcal{U}_{\beta }$ is equivalent to the study of \emph{%
open systems}. Such systems have been studied extensively by physicists and
mathematicians in many aspects. Main questions in open systems are: how do
typical points of the phase space escape from a given hole, what is the
speed of escape, and which hole is leaking the most (see \cite%
{KL-JoSP2009,AB-N2010, BY-IJoM2011, DWY-CiMP2010, DY-N2006})? Escape rate is
introduced to measure such quantity and we define such rate in our setting.
Denote 
\begin{eqnarray*}
\widetilde{\Gamma }_{\beta ,n} &=&\widetilde{\Gamma }_{n}(\Delta _{\beta })
\\
&=&\{x\in J_{\beta }:G_{\beta }^{n}(x)\in \Delta _{\beta }\text{ and }%
G_{\beta }^{k}(x)\notin \Delta _{\beta }\text{ for }0\leq k\leq n-1\}.
\end{eqnarray*}%
Since the trapping region of $(J_{\beta },G_{\beta })$ is $(I,T_{\beta })$,
we can define 
\begin{eqnarray*}
\Gamma _{\beta ,n} &=&\Gamma _{n}(\Delta _{\beta }) \\
&=&\{x\in I:T_{\beta }^{n}(x)\in \Delta _{\beta }\cap I\text{ and }T_{\beta
}^{k}(x)\notin \Delta _{\beta }\cap I\text{ for }0\leq k\leq n-1\}.
\end{eqnarray*}%
It is clear that $\mathcal{U}_{\beta }\cap I=\cap _{n=1}^{\infty
}I\backslash \Gamma _{n}(\Delta _{\beta })$. Since $\lambda (\mathcal{U}%
_{\beta })=0$, we have $\lim\nolimits_{n\rightarrow \infty }\lambda
(I\backslash \Gamma _{n}(\Delta _{\beta }))=0$. (Recall that $\lambda $ is
the Lebesgue measure on $\mathbb{R}$.) Note that $\Gamma _{n+1}(\Delta
_{\beta })\supset \Gamma _{n}(\Delta _{\beta })$ for all $n\geq 0$ , the
following limit exists and we can define the \emph{escape rate} as follows 
\begin{equation*}
e_{\beta }=\lim_{n\rightarrow \infty }\frac{-\log \lambda (\Gamma _{\beta
,n})}{n}\text{ and }E_{\beta }=\lim_{n\rightarrow \infty }\frac{-\log
\lambda (\Gamma _{\beta ,n})}{n\log \beta }=\frac{e_{\beta }}{\log \beta }%
\text{.}
\end{equation*}%
Here we also call $E_{\beta }$ the \emph{escape rate} if it causes no
confusion. The corresponding escape rate $\widetilde{e}_{\beta }$ and $%
\widetilde{E}_{\beta }$ are defined similarly for $\widetilde{\Gamma }%
_{n}(\Delta _{\beta })$.

The aim of this paper is to calculate the rate $e_{\beta }$ and describe how 
$e_{\beta }$ changes as $\beta $ varies. It is worth pointing out that such
a problem has been raised by Bundfuss, Kr\"{u}ger and Troubetzkoy (p.23, 
\cite{BKT-ETaDS2011}).

\begin{quotation}
\textit{One would like to develop a relationship between the escape rate
properties and topological and/or metric invariants of the invariant set.}
\end{quotation}

We emphasize that the central problem of open systems is how to estimate the
escape rate and how the escape rate varies when the hole is shrinking to
zero. Our problem is a little bit different since the map $T_{\beta }$ also
changes with respect to $\beta $ in $(1,2]$. The following is the main
result of this investigation.

\begin{theorem}
\label{Thm: Main}Let $1<\beta \leq 2$. Then $\dim _{H}\mathcal{U}_{\beta
}+E_{\beta }=1$.
\end{theorem}

\begin{corollary}
\label{cordl} Let $1<\beta \leq 2$.

\begin{enumerate}
\item If $1<\beta \leq \beta _{\ast }$, then $E_{\beta }=1$ and $e_{\beta
}=\log \beta $.

\item If $\beta _{\ast }<\beta \leq 2$, then $E_{\beta }$ forms a devil
staircase function (i.e., continuous, monotonic, and the derivatives of $%
E_{\beta }$ with respect to $\beta $ are large than zero almost everywhere).
Moreover, $\lim\nolimits_{\beta \rightarrow \beta ^{\ast }}E_{\beta }=1$ and 
$E_{2}=0$.

\item Let $[\beta _{L,}\beta _{U}]$ be the admissible interval generated by
a block $t_{1}\cdots t_{p}$ (see \cite{KL-N2015}). For $\beta \in \lbrack
\beta _{L,}\beta _{U}]$, the escape rate $E_{\beta }$ is given by 
\begin{equation*}
E_{\beta }=1-\frac{h_{top}(Z_{t_{1}\cdots t_{p}})}{\log \beta },
\end{equation*}%
where $h_{top}(Z_{t_{1}\cdots t_{p}})$ is the entropy of the subshift of
finite type%
\begin{equation}
Z_{t_{1}\cdots t_{p}}:=\{(d_{i}):\overline{t_{1}\cdots t_{p}}\leq
d_{n}\cdots d_{n+p-1}\leq t_{1}\cdots t_{p},\text{ }n\geq 1\}  \label{2}
\end{equation}
and $\overline{\ell} := 1 - \ell$.
\end{enumerate}
\end{corollary}

Some related results are also addressed herein. In \cite{FS-MfuM2011}, Feng
and Sidorov considered the \emph{growth rate} of the points having a
continuum of $\beta $-expansions, which is somehow a kind of duality of
escape rate we study here. Ban \emph{et al.} \cite{BHL-IJBCASE2003}
considered a unimodal map with a symmetric hole in the middle and study the
topological entropy of those points whose orbits never fall in the hole
under iteration. The authors showed that the entropy function forms a
devil's staircase function with respect to the size of the hole. However,
the constant part of such function are not completely characterized.
Misiurewicz \cite{Misiurewicz-IJBCASE2004} also provided an topological
proof for the same result. If $\beta =2$, Barrera \cite{Barrera-apa2013} put
a symmetry hole about $1/2$ and show that the entropy function forms a
devil's staircase with respect to the the size of the hole, and the author
completely characterized the constant part of the entropy (also called the $%
\emph{entropy}$ $\emph{pleateau}$). The topics of the transitive components
of the open systems are discussed in (\cite%
{Barrera-apa2013,Barrera-2014,BKT-ETaDS2011}).

Section 2 is devoted to the proof of Theorem \ref{Thm: Main}.

\section{Proof of Theorem \protect\ref{Thm: Main}}

Before proving the main theorem, we provide some useful materials on open
systems. Let $f:M\rightarrow M$ be a map which admits a Markov partition $%
\mathcal{P}$. We denote by $\pi :\Sigma ^{\mathcal{P}}\rightarrow M$ the
corresponding coding map and $\Sigma ^{\mathcal{P}}$ the corresponding
symbolic space with respect to $\mathcal{P}$. Fix an open hole $H\subset M$,
set $\Lambda ^{\ast }=\Lambda _{H}^{\ast }$ the invariant set of points
whose orbits never fall in the hole under $f$. Denote by $\Sigma ^{\ast
}=\Sigma _{H}^{\ast }=\pi ^{-1}\Lambda ^{\ast }$ the preimage of $\Lambda
^{\ast }$ under $\pi $ and denote by $\sigma $ the shift map on $\Sigma
^{\ast }$. Let $\partial H$ be the boundary of $H$. The following result
shows that once the boundary points fall in the gap $H$ under iteration of $%
f $, then $\Sigma ^{\ast }$ is a subshift of finite type.

\begin{proposition}[Proposition 4.1 \protect\cite{BKT-ETaDS2011}]
\label{Prop: 3}If for each $x\in \partial H$ there is an $i$ such that $%
f^{i}x\in H$, then $\Sigma ^{\ast }$ is a SFT.
\end{proposition}

\begin{remark}
\label{Rmk: 1}

\begin{enumerate}
\item It is worth noting that $\Delta _{\beta }$ is not open. However, the
computation of the Hausdorff dimension and escape rate of $\mathcal{U}_{q}$
are not affected if we substitute $\Delta _{\beta }=(1/\beta ,1/\beta (\beta
-1))$ since $\overline{\mathcal{U}}_{q}\backslash \mathcal{U}_{q}$ is
countable \cite{KL-JoNT2007}, where $\overline{\mathcal{U}}_{q}$ is the
closure of $\mathcal{U}_{q}$. Thus we define $\Delta _{\beta }$ as such an
open interval in what follows.

\item We point out that there is an analogous result of Proposition \ref%
{Prop: 3} in IFS setting (Theorem 2.4, \cite{BDJ-apa2014}). The authors in 
\cite{BDJ-apa2014} also show that the set of points whose orbits fall in
holes is of full Lebesgue measure (Corollary 3.1, \cite{BDJ-apa2014}).
\end{enumerate}
\end{remark}

The following simple proposition reveals that the escape rate are the same
under the dynamical systems $(J_{\beta },G_{\beta })$ and $(I,T_{\beta })$,
and the proof is omitted.

\begin{proposition}
Let $1<\beta \leq 2$. Then $e_{\beta }=\widetilde{e}_{\beta }\ \ \ $and$\ \
\ E_{\beta }=\widetilde{E}_{\beta }$.
\end{proposition}

We denote by $\Lambda _{\beta }^{\ast }$ the collection of points whose
orbits never fall in the gap $\Delta _{\beta }$. Note that $\Lambda _{\beta
}^{\ast }$ is $T_{\beta }$-invariant. Therefore, $(\Lambda _{\beta }^{\ast
},T_{\beta }^{\ast })$ is a dynamical system on its own right, where $%
T_{\beta }^{\ast }=T_{\beta }|_{\Lambda _{\beta }^{\ast }}$. Also we denote
by $\Sigma _{\beta }^{\ast }=\pi _{\beta }^{-1}\Lambda _{\beta }^{\ast }$
the symbolic space. Let $a=a_{\beta }:=1/\beta $ and $b=b_{\beta }:=1/\beta
(\beta -1)$. Define $\mathcal{F}:=\{\beta \in (1,2]:$ $a$ and $b$ fall in $%
\Delta _{\beta }$ under iteration$\}$ and $\mathcal{N}=I\backslash \mathcal{F%
}$, i.e., $a$ or $b$ does not fall in the hole $\Delta _{\beta }$ under $%
T_{\beta }$ for $\beta \in \mathcal{N}$. Proposition \ref{Prop: 3} shows
that $\Sigma _{\beta }^{\ast }$ is a subshift of finite type for $\beta \in 
\mathcal{F}$.

\subsection{The case where $\protect\beta \in \mathcal{F}$}

In this section, we discuss the case where $\beta \in \mathcal{F}$. The aim
of this section is to define a new map $\mathbf{T}_{\beta }$ which enables
us to apply the results of Afraimovich and Bunimovich \cite{AB-N2010} on
open systems. Introduce a piecewise linear map $\mathbf{T}_{\beta
}:I\rightarrow I$ from $T_{\beta }$ as follows.

\begin{equation*}
\mathbf{T}_{\beta }(x)=\left\{ 
\begin{array}{ll}
T_{\beta }(x), & \text{if }x\notin \Delta _{\beta }; \\ 
x, & \text{if }x\in \Delta _{\beta }.%
\end{array}%
\right.
\end{equation*}%
That is, those points which fall in $\Delta _{\beta }$ under iteration of $%
T_{\beta }$ are \textbf{stuck} by $\mathbf{T}_{\beta }$.

For $\beta \in \mathcal{F}$, let $i_{\ast }\geq 1$ be such that $T_{\beta
}^{i_{\ast }}\ast \in \Delta _{\beta }$ and $T_{\beta }^{i}\ast \notin
\Delta _{\beta }$ for $0\leq i<i_{\ast }$, where $\ast $ stands for $a$ or $%
b $. That is, $i_{\ast }$ is the first return time of the orbits of $\ast $
falls in the hole $\Delta _{\beta }$. Denote by $\mathcal{A}_{\beta
}=\{T_{\beta }^{k}a\}_{k=0}^{i_{a}}$ and $\mathcal{B}_{\beta }=\{T_{\beta
}^{k}b\}_{k=0}^{i_{b}}$. Let $\mathcal{C}_{\beta }=\mathcal{A}_{\beta }\cup 
\mathcal{B}_{\beta }$ be an ordered set in $\mathbb{R}$. Then $\mathcal{C}%
_{\beta }$ is a partition of $I$. For $\beta \in \mathcal{F}$, the following
lemma provides the explicit representation of the Markov partition of $%
\mathbf{T}_{\beta }$.

\begin{lemma}
\label{Lma: 1}For $\beta \in \mathcal{F}$, $\mathcal{C}_{\beta }$ is a
Markov partition for $\mathbf{T}_{\beta }$.
\end{lemma}

\begin{proof}
If suffices to show that $\mathbf{T}_{\beta }z\in \mathcal{C}_{\beta }$ for
all $z\in \mathcal{C}_{\beta }$. We claim that, if $z\in \mathcal{C}_{\beta
} $ and $T_{\beta }z\notin \Delta _{\beta }$, then $\mathbf{T}_{\beta
}z=T_{\beta }z\in \mathcal{C}_{\beta }.$ If this is not the case, then $%
T_{\beta }z\in \Delta _{\beta }$ implies $\mathbf{T}_{\beta }z=z\in \mathcal{%
C}_{\beta }$, which completes the proof.
\end{proof}

For $\beta \in \mathcal{F}$, let $\eta =\{\eta _{j}\}_{j\in \mathcal{I}}$ be
the Markov partition of $\mathbf{T}_{\beta }$ (Lemma \ref{Lma: 1}).
Decompose the index set $\mathcal{I}=\mathcal{I}_{H}\cup \mathcal{I}_{0}$,
where $\mathcal{I}_{H}=\left\{ i\in \mathcal{I}:\eta _{i}\subseteq \Delta
_{\beta }\right\} $ and $\mathcal{I}_{0}=\mathcal{I}$ $\backslash $ $%
\mathcal{I}_{H}$, and let $A_{\beta }$ be the corresponding transition
matrix, i.e., $A_{\beta }(i,j)=1$ if $\mathbf{T}_{\beta }\eta _{i}\subseteq
\eta _{j}$ and $A_{\beta }(i,j)=0$ otherwise. Denote by $A_{\beta }^{-}$ the 
$0$-$1$ matrix which is derived by deleting the $\mathcal{I}_{H}$-columns
and rows of $A_{\beta }$. Let $X_{\beta }^{-}=X_{A_{\beta }^{-}}$ be the
subshift generated by the adjacency matrix $A_{\beta }^{-}$. Define $\Theta :%
\mathcal{I}\rightarrow \{0,1\}$ by $\Theta (i)=0$ if $\eta _{i}\subseteq
I_{0}$ and $\Theta (i)=1$ if $\eta _{i}\subseteq I_{1}$. Let $\theta
:X_{\beta }^{-}\rightarrow \{0,1\}^{\mathbb{N}}$ be the map induced by $%
\Theta ,$ i.e., $\theta (\omega )=(\Theta (\omega _{1}),\Theta (\omega
_{2}),\ldots )$, where $\omega =(\omega _{1},\omega _{2},\ldots )\in
X_{\beta }^{-}$. Since the indices of $A_{\beta }^{-}$ are those intervals
of $\mathcal{I}_{0}=\mathcal{I}$ $\backslash $ $\mathcal{I}_{H}$, it is
evident that 
\begin{equation}
h_{top}(X_{\beta }^{-})=h_{top}(\Sigma _{\beta }^{\ast })\text{.}  \label{1}
\end{equation}

Let $M$ be a square matrix. We denote by $\rho _{M}$ the maximal eigenvalue
of $M$. Define $B_{\beta }=A_{\beta }\times $ diag$\left( 1/\beta ,\ldots
,1/\beta \right) $, and $B_{\beta }^{-}$ is derived by deleting the $%
\mathcal{I}_{H}$-columns and rows of $B_{\beta }$. The following result in 
\cite{AB-N2010} links the escape rate $e_{\beta }$ with the entropy $%
h_{top}(\Sigma _{\beta }^{\ast })$.

\begin{lemma}[Theorem 4, \protect\cite{AB-N2010}]
\label{Lma: 3}The measure $\lambda (\Gamma _{\beta ,n})$ satisfies the
following asymptotic equality: 
\begin{equation*}
\lambda (\Gamma _{\beta ,n})\simeq Q_{0}(n)\rho _{B_{\beta }^{-}}^{n}\lambda
(\Delta _{\beta })\text{,}
\end{equation*}%
where $Q_{0}(n)$ is a polynomial with degree less then the number of the
holes $m$ (herein $m=1$).
\end{lemma}

\begin{proof}[Proof of Theorem \protect\ref{Thm: Main} for $\protect\beta %
\in \mathcal{F}$]
Suppose $\beta \in \mathcal{F}$. Since $\rho _{B_{\beta }^{-}}=\beta
^{-1}\rho _{A_{\beta }^{-}},$ Lemma \ref{Lma: 3} is applied to show that 
\begin{eqnarray*}
e_{\beta } &=&\lim_{n\rightarrow \infty }\frac{-\log \lambda (\Gamma _{\beta
,n})}{n}=-\log \rho _{B_{\beta }^{-}}=\log \beta -\log \rho _{A_{\beta }^{-}}
\\
&=&\log \beta -h_{top}(X_{\beta }^{-})=\log \beta -h_{top}(\Sigma _{\beta
}^{\ast })\text{.}
\end{eqnarray*}%
Since $\dim _{H}\mathcal{U}_{\beta }=\frac{h_{top}(\Sigma _{\beta }^{\ast })%
}{\log \beta }$ (Theorem 1.3, \cite{KKL-apa2015}), we have 
\begin{equation*}
E_{\beta }=\frac{e_{\beta }}{\log \beta }=1-\frac{h_{top}(\Sigma ^{\ast })}{%
\log \beta }=1-\dim _{H}\mathcal{U}_{\beta }\text{.}
\end{equation*}%
This completes the proof for $\beta \in \mathcal{F}$.
\end{proof}

\subsection{The case where $\protect\beta \in \mathcal{N}$}

\begin{proof}[Proof of Theorem \protect\ref{Thm: Main}]
Let $\beta \in \mathcal{N}$, $\Lambda ^{\ast }=\Lambda _{\beta }^{\ast }$, $%
T^{\ast }=T_{\beta }^{\ast },e=e_{\beta }$, and $E=E_{\beta }$. The idea of
this proof is to approximate $(\Lambda ^{\ast },T^{\ast })$ by the open
systems $\{(\Lambda _{l}^{\ast },T_{l}^{\ast })\}_{l=1}^{\infty }$ such that 
$\Sigma _{l}^{\ast }=\pi _{\beta }^{-1}\Lambda _{l}^{\ast }$ is a SFT for
all $l\geq 1$.

It suffices to prove the case of $T^{n}a\notin \Delta _{\beta }$ for all $%
n\geq 1$ and $b$ fall in $\Delta _{\beta }$ under iteration of $T$. The same
argument remains valid for other cases. Since the set of points whose orbits
fall in the hole $\Delta _{\beta }$ is of full Lebesgue measure (Remark \ref%
{Rmk: 1} (ii)), we construct two sequences $\{a_{l}\}_{l=1}^{\infty
}\subseteq \mathbb{R}$ and $\{i_{l}\}_{l=1}^{\infty }\subseteq \mathbb{N}$
such that $a_{l}\leq a_{l+1}$, $\lim_{l\rightarrow \infty }a_{l}=a=\frac{1}{%
\beta }$, $T^{i_{l}}a_{l}\in \Delta _{\beta }$, and $T^{i}a_{l}\notin \Delta
_{\beta }$ for $1\leq i<i_{l}$. Denote by $\Lambda _{l}^{\ast }$ the set of
points whose orbits never fall in $\Delta _{l}:=(a_{l},\frac{1}{\beta (\beta
-1)})$. Thus we have $\Lambda _{l}^{\ast }\subseteq \Lambda _{l+1}^{\ast }$
and $\Lambda _{\beta }^{\ast }=\overline{\cup _{l=1}^{\infty }\Lambda
_{l}^{\ast }}$. Define $\Sigma _{l}^{\ast }=\pi _{\beta }^{-1}\Lambda
_{l}^{\ast }$, Proposition \ref{Prop: 3} infers that $\Sigma _{l}^{\ast }$
is a SFT (sinec $a_{l}$ and $b=\frac{1}{\beta (\beta -1)}$ fall in $\Delta
_{l}$). Construct $A_{l},A_{l}^{-},B_{l},B_{l}^{-}$, and $X_{l}^{-}$
analogously to the case where $\beta \in \mathcal{F}$. Let $\rho _{l}=\rho
_{B_{l}^{-}}$. We then have $h_{top}(\Lambda _{l})=h_{top}(\Sigma _{l})=\log
\rho _{l}$ under the same discussion of (\ref{1}).

Clearly, $\rho _{l+1}\geq \rho _{l}$. We claim that the sequence $%
\{a_{l}\}_{l=1}^{\infty }$ and $\{i_{l}\}_{l=1}^{\infty }$ can be chosen so
that $\rho _{l+1}>\rho _{l}$. Since $\Sigma _{l}^{\ast }$ is a SFT, Theorem
6.4 in \cite{BKT-ETaDS2011} shows that the number of topologically
transitive components of $\left( \Lambda _{l}^{\ast }\right) ^{\text{nw}}$
of $\Lambda _{l}^{\ast }$ is at most $2$ (it is $2r$ actually, where $r$ is
the number of holes, and $r=1$ in our case), where \textquotedblleft
nw\textquotedblright\ stands for the non-wandering set of $\Lambda
_{l}^{\ast }$. Since $0$ is a trivial transitive component, thus the number
of non-trivial proper topological transitive components with uncountable
elements is exact one.

Once the pair $\left( a_{l},i_{l}\right) $ has been chosen for $l\geq 1$, we
pick a pair $(a_{l+1},i_{l+1})$ such that $T^{i_{l+1}}a_{l+1}\in \Delta
_{l+1}$ but $T^{i_{l}}a_{l+1}\notin \Delta _{l}$. That is, $a_{l+1}$ does
not fall in the hole $\Delta _{l}$ for the first $i_{l}$ iterations. This is
possible since $\beta >1$ and the function $a\rightarrow T^{i_{l}}a$ grows
fast with respect to $a$ if $i_{l}\geq 1$ is large enough. Therefore, $%
a_{l+1}$ can be chosen among such those points and wait for its orbit fall
in the hole $\Delta _{l+1}$ again. For each $l\geq 1$, $T_{l}^{\ast }$
admits a Markov partition and is topologically transitive. Let $A_{l}^{-}$
and $A_{l+1}^{-}$ be the corresponding adjacency matrices. We may assume
that $A_{l}^{-}$ and $A_{l+1}^{-}$ are of the same size, otherwise one can
present both of them by $N$th higher block representation (Definition 1.4.1, 
\cite{LM-1995}) until the lengths of all forbidden sets in $X_{l}^{-}$ and $%
X_{l+1}^{-}$ are less than $N$ for some $N\in \mathbb{N}$. From the
construction of $(a_{l+1},i_{l+1}) $ above we know that there exists at
least a word which belongs to $X_{l+1}^{-}$ but not belong to $X_{l}^{-}$.
That is, $A_{l}^{-}<A_{l+1}^{-}$. Thus, we have $B_{l}^{-}<B_{l+1}^{-}$.
Theorem 4.4.7 in \cite{LM-1995} indicates that $\rho _{l}<\rho _{l+1}$.

Since $\Lambda ^{\ast }=\overline{\cup _{l=1}^{\infty }\Lambda _{l}^{\ast }}$%
, it follows from Lemma 4.1.10 in \cite{ALM-2000} (see Remark \ref{Rmk: 2}
below) that 
\begin{equation*}
h_{top}(\Lambda ^{\ast })=\sup_{l}h_{top}(\Lambda _{l}^{\ast
})=\lim_{l\rightarrow \infty }h_{top}(\Lambda _{l}^{\ast }).
\end{equation*}

Let $e_{l}$ (resp. $E_{l}$) be the escape rate corresponding to the hole $%
\Delta _{l}$. From Lemma \ref{Lma: 3}, we deduce that, for all $l\geq 1$, 
\begin{equation*}
e_{l}=\log \beta -\log \rho _{l}\text{ and }E_{l}=\frac{e_{l}}{\log \beta }%
=1-\frac{\log \rho _{l}}{\log \beta }=1-\frac{h_{top}(\Lambda _{l}^{\ast })}{%
\log \beta }\text{.}
\end{equation*}%
Taking $l\rightarrow \infty $, then we have%
\begin{equation*}
E=\lim_{l\rightarrow \infty }E_{l}=1-\frac{1}{\log \beta }\lim_{l\rightarrow
\infty }h_{top}(\Lambda _{l}^{\ast })=1-\frac{1}{\log \beta }h_{top}(\Lambda
^{\ast })=1-\dim _{H}\mathcal{U}_{\beta }\text{.}
\end{equation*}%
The last equality comes from the fact that $\dim _{H}\mathcal{U}_{\beta }=%
\frac{h_{top}(\Lambda ^{\ast })}{\log \beta }$ for general subshift (Theorem
1.3, \cite{KKL-apa2015}). That is, $\dim _{H}\mathcal{U}_{\beta }+E=1$,
which establishes the formula.
\end{proof}

\begin{remark}
\label{Rmk: 2} Lemma 4.1.10 in \cite{ALM-2000} shows that if $X=\cup
_{i=1}^{k}Y_{i}$, and the sets $Y_{i}$ are closed and invariant, then $%
h_{top}(X)=\max_{i}h_{top}(Y_{i})$. However, the Theorem stays true if we
replace a finite cover of $X$ by an infinite one (see \cite{ALM-2000} for
more details).
\end{remark}

\begin{proof}[Proof of Corollary \protect\ref{cordl}]
(1) and (3) of Corollary \ref{cordl} are the immediate consequences of the
following facts: (i) $\dim _{H}\mathcal{U}_{\beta }=0$ for $1<\beta \leq
\beta _{\ast }$, (ii) $\dim _{H}\mathcal{U}_{2}=1$, and (iii) $\dim _{H}%
\mathcal{U}_{\beta }=\frac{h_{top}(Z_{t_{1}\cdots t_{p}})}{\log \beta }$ for 
$\beta \in \lbrack \beta _{L,}\beta _{U}]$, where $h_{top}(Z_{t_{1}\cdots
t_{p}})$ is the topological entropy of the SFT (\ref{2}) (Theorem 2.6, \cite%
{KL-N2015}). Finally, combining Theorem \ref{Thm: Main} with the fact that $%
q\rightarrow \dim _{H}\mathcal{U}_{\beta }$ forms a devil staircase function
(Theorem 1.7, \cite{KKL-apa2015}) yields (2). This completes the proof.
\end{proof}

\bibliographystyle{amsplain}
\bibliography{ban}

\providecommand{\bysame}{\leavevmode\hbox to3em{\hrulefill}\thinspace}
\providecommand{\MR}{\relax\ifhmode\unskip\space\fi MR }
\providecommand{\MRhref}[2]{%
  \href{http://www.ams.org/mathscinet-getitem?mr=#1}{#2}
}
\providecommand{\href}[2]{#2}
\begin{thebibliography}{10}

\bibitem{AB-N2010}
V.~S. Afraimovich and L.~A. Bunimovich, \emph{Which hole is leaking the most: a
  topological approach to study open systems}, Nonlinearity \textbf{23} (2010),
  no.~3, 643--656.

\bibitem{ALM-2000}
L.~Alsed\`{a}, J.~Llibre, and M.~Misiurewicz, \emph{Combinatorial dynamics and
  entropy in dimension one}, Advanced Series in Nonlinear Dynamics, 5. World
  Scietific, Singapore, 2000.

\bibitem{BDJ-apa2014}
S.~Baker, K.~Dajani, and K.~Jiang, \emph{On univoque points for self-similar
  sets}, arXiv preprint arXiv:1406.3263 (2014).

\bibitem{BHL-IJBCASE2003}
J.-C. Ban, C.-H. Hsu, and S.-S. Lin, \emph{Devil's staircase of gap maps},
  Internat. J. Bifur. Chaos Appl. Sci. Engrg. \textbf{13} (2003), no.~01,
  115--122.

\bibitem{Barrera-apa2013}
R.~A. Barrera, \emph{Topological and ergodic properties of symmetric
  subshifts}, arXiv preprint arXiv:1306.2054 (2013).

\bibitem{Barrera-2014}
\bysame, \emph{Topological and symbolic dynamics of the doubling map with a
  hole}, Ph.D. thesis, University of Manchester, 2014.

\bibitem{BKT-ETaDS2011}
S.~Bundfuss, T.~Kr\"{u}ger, and S.~Troubetzkoy, \emph{Topological and symbolic
  dynamics for hyperbolic systems with holes}, Ergodic Theory Dynam. Systems
  \textbf{31} (2011), no.~5, 1305--1323.

\bibitem{BY-IJoM2011}
L.~Bunimovich and A.~Yurchenko, \emph{Where to place a hole to achieve a
  maximal escape rate}, Israel J. Math. \textbf{182} (2011), no.~1, 229--252.

\bibitem{DWY-CiMP2010}
M.~Demers, P.~Wright, and L.-S. Young, \emph{Escape rates and physically
  relevant measures for billiards with small holes}, Comm. Math. Phys.
  \textbf{294} (2010), no.~2, 353--388.

\bibitem{DY-N2006}
M.~Demers and L.-S. Young, \emph{Escape rates and conditionally invariant
  measures}, Nonlinearity \textbf{19} (2006), no.~2, 377--397.

\bibitem{EJK-BdlSMdF1990}
P.~Erd{\"o}s, I.~Jo{\'o}, and V.~Komornik, \emph{Characterization of the unique
  expansions $1=\sum^{\infty}_{i= 1}q^{-n_ i}$ and related problems}, Bulletin
  de la Soci{\'e}t{\'e} Math{\'e}matique de France \textbf{118} (1990), no.~3,
  377--390.

\bibitem{FS-MfuM2011}
D.-J. Feng and N.~Sidorov, \emph{Growth rate for beta-expansions}, Monatshefte
  f{\"u}r Mathematik \textbf{162} (2011), no.~1, 41--60.

\bibitem{GS-MRL2001}
P.~Glendinning and N.~Sidorov, \emph{Unique representations of real numbers in
  non-integer bases}, Math. Res. Lett. \textbf{8} (2001), no.~4, 535--544.

\bibitem{KL-JoSP2009}
G.~Keller and C.~Liverani, \emph{Rare events, escape rates and
  quasistationarity: some exact formulae}, J. Stat. Phys. \textbf{135} (2009),
  no.~3, 519--534.

\bibitem{KKL-apa2015}
V.~Komornik, D.-R. Kong, and W.-X. Li, \emph{Hausdorff dimension of univoque
  sets and devil's staircase}, arXiv preprint arXiv:1503.00475 (2015).

\bibitem{KL-AMM1998}
V.~Komornik and P.~Loreti, \emph{Unique developments in non-integer bases},
  Amer. Math. Monthly (1998), 636--639.

\bibitem{KL-JoNT2007}
V.~Komornik and P.~Loreti, \emph{On the topological structure of univoque
  sets}, Journal of Number Theory \textbf{122} (2007), no.~1, 157--183.

\bibitem{KL-N2015}
D.-R. Kong and W.-X. Li, \emph{Hausdorff dimension of unique beta expansions},
  Nonlinearity \textbf{28} (2015), no.~1, 187--209.

\bibitem{LM-1995}
D.~Lind and B.~Marcus, \emph{An introduction to symbolic dynamics and coding},
  Cambridge University Press, Cambridge, 1995.

\bibitem{Misiurewicz-IJBCASE2004}
M.~Misiurewicz, \emph{Entropy of maps with horizontal gaps}, Internat. J.
  Bifur. Chaos Appl. Sci. Engrg. \textbf{14} (2004), no.~04, 1489--1492.

\bibitem{Renyi-AMH1957}
A.~R{\'e}nyi, \emph{Representations for real numbers and their ergodic
  properties}, Acta Math. Hungar. \textbf{8} (1957), no.~3, 477--493.

\bibitem{Sidorov-AMM2003}
N.~Sidorov, \emph{Almost every number has a continuum of $\beta$-expansions},
  Amer. Math. Monthly (2003), 838--842.

\end{thebibliography}


\end{document}